\documentclass[12pt]{amsart}

\usepackage{amsmath}
\usepackage{amssymb,latexsym,cite}
\usepackage{amscd}
\usepackage{comment}
\usepackage{xspace}
\usepackage{verbatim}
\usepackage{todonotes}
\usepackage{hyperref}
\usepackage[foot]{amsaddr}

\newtheorem{theorem}{Theorem}[section]
\newtheorem{proposition}[theorem]{Proposition}

\newtheorem{lemma}[theorem]{Lemma}
\theoremstyle{definition}
\newtheorem{remark}[theorem]{Remark}
\newtheorem{definition}[theorem]{Definition}
\numberwithin{equation}{section}

\newcommand{\R}{{\mathbb R}}

\newcommand{\tr}{\mathrm{tr}^*}

\newcommand{\chr}[1]{\mathbf{1}\ind{#1}}

\newcommand{\BSA}{\begin{subarray}}
	\newcommand{\ESA}{\end{subarray}}
\newcommand{\BAL}{\begin{aligned}}
	\newcommand{\EAL}{\end{aligned}}

\newcommand{\note}[1]{\noindent\textit{#1.}\hspace{2mm}}
\newcommand{\Remark}{\note{Remark}}
\newcommand{\forevery}{\quad \forall}

\newcommand{\norm}[1]{\left \|#1\right \|}

\newcommand{\rec}[1]{\frac{1}{#1}}

\newcommand{\supp}{\mathrm{supp}\,}
\newcommand{\dist}{\mathrm{dist}\,}
\newcommand{\sign}{\mathrm{sign}\,}

\newcommand{\prt}{\partial}

\newcommand{\ti}{\times}

\newcommand{\tl}{\tilde}

\newcommand{\sbs}{\subset}

\newcommand{\ind}[1]{_{_{#1}}}

\newcommand{\wrto}{with respect to\xspace}

\newcommand{\sth}{such that\xspace}

\newcommand{\bvp}{boundary value problem\xspace}

\newcommand{\bdw}{\partial\Gw}

\newcommand{\qtxt}[1]{\quad\textrm{#1}}
\newcommand{\ntxt}[1]{\noindent\textit{#1}.}

\def\ga{\alpha}     \def\gb{\beta}       \def\gg{\gamma}
       \def\gd{\delta}      \def\ge{\epsilon}
\def\gth{\theta}                         
           
            \def\gl{\lambda}

\def\gs{\sigma}       
      \def\gw{\omega}
\def\gx{\xi}                \def\gz{\zeta}
     \def\Gd{\Delta}      

    \def\Gs{\Sigma}      
\def\Gw{\Omega}              

\def\BBG {\mathbb G}       
   \def\BBK {\mathbb K}    
       
   \def\BBR {\mathbb R}

\def\GTM {\mathfrak M}

\def\btr{boundary trace\xspace}
\def\tr{\mathrm{tr\,}}

\def\LVsup{$L_V$ superharmonic\xspace}
\def\LVsub{$L_V$ subharmonic\xspace}
\def\LVhr{$L_V$ harmonic\xspace}

\def\q{\quad}

\def\GPV{\Gw;\Phi_V}
\def\GPV0{\Gw,\frac{\Phi_V}{Phi_0}}

\begin{document}
	
	\title[Semilinear Schr\"odinger equations]{Boundary value problems for semilinear Schr\"odinger equations with singular potentials and measure data}
	\author[M. Bhakta]{Mousomi Bhakta}
\address[M. Bhakta]{Department of Mathematics, Indian Institute of Science Education and Research, Pune 411008, India}
\email[M. Bhakta]{mousomi@iiserpune.ac.in}
\author[M. Marcus]{Moshe Marcus} 
\address[M. Marcus, Corresponding author]{Department of Mathematics, Technion\\
 Haifa 32000, Israel}
 \email{marcusm@math.technion.ac.il}
 \author[P. T. Nguyen]{ Phuoc-Tai Nguyen}
\address[P. T. Nguyen]{Department of Mathematics and Statistics, Masaryk University, Brno, Czech Republic}
\email[P. T. Nguyen]{ptnguyen@math.muni.cz}


	\begin{abstract}
	We study boundary value problems with measure data in smooth bounded domains $\Gw$,  for semilinear equations. 
	Specifically we consider problems of the form
	$-L_Vu + f(u) = \tau$ in $\Gw$ and $\tr_V u=\nu$ on $\bdw$, where $L_V= \Gd+V$, $f\in C(\BBR)$ is monotone increasing with $f(0)=0$ and $\tr_V u$ denotes the measure boundary trace of $u$ associated with $L_V$. 
	The potential $V$ is typically a H\"older continuous function in $\Gw$  that blows up at a set $F\sbs \bdw$ as $\dist(x,F)^{-2}$. In general the above \bvp may not have a solution. We are interested in questions related to the concept of `reduced measures', introduced in \cite{BMP} for $V=0$.
	Our results extend results of \cite{BMP} and \cite{BP} and apply to a larger class of nonlinear terms
	$f$. In the case of signed measures, some of the present results are new even for $V=0$. 
	
	\vskip 3mm
	
	
	\vskip 2mm
	

	\end{abstract}
	
	
	\maketitle
	
	\tableofcontents
	
	\section{Introduction}
Let $\Gw$ be a $C^2$ bounded domain in $\BBR^N$, $N\geq 3$, and let 
$$L_V:=\Gd+V$$ where $V \in C^\gth(\Gw)$, for some $\gth\in (0,1]$, satisfies the following conditions:
\begin{equation} \label{A1} \tag{A1} \exists\, \bar a>0\, :\quad |V(x)| \leq \bar a \gd(x)^{-2} \forevery\, x\in \Gw, 
\end{equation}
$$ \gd(x):=\dist(x,\bdw),$$
\begin{equation} \label{A2} \tag{A2}
	\int_\Gw |\nabla\phi|^2\,dx\geq \int_\Gw \phi^2 V\, dx  \forevery\, \phi\in H^1_0(\Gw). 
\end{equation}
These conditions imply the existence of a (minimal) Green function $G_V$ and of the Martin kernel $K_V$ for the operator $-L_V$. Related to this, the operator has a ground state that we denote by $\Phi_V$. In the present case $\Phi_V$ is a positive eigenfunction of $-L_V$ with eigenvalue $\gl_V>0$.

The function $\Phi_V$ and the Martin kernel $K_V$ are normalized at a reference point $x_0\in\Gw$: 
$$\Phi_V(x_0)=1,\q K_V(x_0,y)=1 \forevery y\in \bdw.$$

\ntxt{Notation} Denote
\begin{align*}
	\BBK_V[\nu](x):=&\int_{\bdw} K_V(x, y)d\nu(y) \forevery \nu\in \GTM(\bdw),\\
	\BBG_V[\tau](x):= & \int_\Gw G_V(x, y)d\tau(y) \forevery \tau\in \GTM(\Gw;\Phi_V).
\end{align*}
Here $\GTM(\bdw)$ denotes the space of finite Borel measures on $\bdw$ and $\GTM(\Gw;\Phi_V)$
denotes the space of real Borel measures $\tau$ in $\Gw$ \sth $\int_\Gw \Phi_V\,d|\tau|<\infty$. As usual $\GTM_+(\partial \Omega)$ and $\GTM_+(\Omega;\Phi_V)$ denote the positive cones of these spaces. 

A function $u \in L^1_{\rm loc}(\Omega)$ is $L_V$ harmonic (resp. subharmonic, superharmonic) in $\Omega$ if $-L_V u = (\text{resp. } \leq, \, \geq ) \, 0$ in $\Omega$ in the distribution sense.

By the Martin representation theorem, for every positive $L_V$ harmonic function $u$ in $\Gw$ there exists $\nu\in \GTM_+(\bdw)$ \sth $ u=\BBK_V[\nu]$.

By the Riesz decomposition lemma, a positive \LVsup $u$ can be represented in the form
$u=p+h$ where $h$ is the largest \LVhr function dominated by $u$ and $p$ is an $L_V$ potential, i.e. a positive \LVsup function which does not dominate any positive \LVhr.

A function $u$ is an $L_V$ potential if and only if there exists a positive measure $\tau \in \GTM(\Gw;\Phi_V)$ \sth $u=\BBG_V[\tau]$. If $\tau$ is a positive Radon measure then either $\BBG_V[\tau]$ is finite everywhere in $\Gw$ or $\BBG_V[\tau]\equiv \infty$. Moreover, $\BBG_V[\tau]<\infty$ if and only if $\tau\in \GTM(\Gw;\Phi_V)$.

For these and other basic potential theory results we refer the reader to \cite{An88}. A brief survey can  be found in \cite{MM-Green}.
\vskip 2mm

In this paper we study \bvp{s} of the form
\begin{equation}\label{bvp-gen} 
	-L_Vu+f(u)=\tau \qtxt{in }\Gw, \quad 
	\tr_V u=\nu \qtxt{on }\bdw,
\end{equation}
always assuming that 
\begin{equation}\label{f-cond}
	f\in C(\BBR),\qtxt{$f$ is non-decreasing}, \q f(0)=0
\end{equation}
and
$$
	\nu\in \GTM(\bdw), \q \tau\in \GTM(\Gw;\Phi_V).
$$
Finally $\tr_V u$, the $L_V$ boundary trace of $u$, is defined as follows. 

\begin{definition}\label{hmtr}  A non-negative Borel function $u$ defined in $\Gw$ has an $L_V$ \emph{boundary trace} $\nu\in \GTM(\bdw)$ if
	\begin{equation}\label{LVtrace}
		\lim_{n\to\infty}\int_{\prt D_n}hu \,d\gw_V^{x_0,D_n}= \int_{\bdw} hd\nu \forevery\, h\in C(\bar\Gw),
	\end{equation}
	for every uniformly Lipschitz exhaustion $\{D_n\}$ of $\Gw$  \sth $x_0\in D_n$ for all $n$. Here $x_0$ is the reference point previously mentioned and $\gw_V^{x_0,D_n}$ denotes the harmonic measure for $L_V$ in $D_n$ relative to $x_0$. The $L_V$ boundary trace of $u$ is denoted by $\tr_V u$.
	
A real Borel function $u$ defined in $\Gw$ has an $L_V$ boundary trace if
\begin{equation}\label{int|u|}
 \sup_{n}\int_{\partial D_n} |u|\,d\gw_V^{x_0,D_n} < + \infty	
\end{equation} 
and \eqref{LVtrace} holds.	
\end{definition}
\vskip 1mm

When $V=0$ this definition reduces to the classical definition of measure boundary trace.
Recall that 
$$
	d\gw_n^{x_0}=P_{V,n}(x_0, \cdot)dS \qtxt{on }\prt D_n,
$$
where $P_{V,n}$ is the Poisson kernel of $-L_V$ in $D_n$. 

By \cite[Lemma 2.3]{MM-note}, if \eqref{A1}, \eqref{A2} hold,  the $L_V$ trace has the following properties:
\begin{equation}\label{tr-prop'}
	\BAL
	(i)\q \tr_V(\BBK_V[\nu])&=\nu \forevery \nu\in \GTM(\bdw),\\
	(ii)\q \tr_V(\BBG_V[\tau])&=0 \forevery \tau\in \GTM(\Gw;\Phi_V).
	\EAL
\end{equation}
\vskip  1mm
\noindent\textit{Notation.}\hskip 2mm  (i) Let $(\gl,\gs)$ and $(\tau,\nu)$ be  two couples of measures in $\GTM(\Gw;\Phi_V)\ti \GTM(\bdw)$. Then 
$(\gl,\gs) \prec (\tau,\nu)$ means $\gl\leq\tau$ and $\gs\leq \nu$.

(ii) For $\beta>0$, denote
\begin{equation} \label{Omegab} \begin{aligned} D_\beta:=&\{x \in \Omega: \delta(x)>\beta\}, \quad \Omega_\beta:=\{x \in \Omega: \delta(x)<\beta \}, \\ 
&\hskip 1cm \Sigma_\beta:=\{ x \in \Omega: \delta(x)=\beta\}.		
\end{aligned} \end{equation}

Since $\Omega$ is a $C^2$ bounded domain, there exists $\beta_0>0$ such that for every $x \in \Omega_{\beta_0}$ there is a unique point $\sigma(x) \in \partial \Omega$ such that $|x-\sigma(x)|=\delta(x)$, and $x \mapsto \delta(x)$ is in $C^2(\Omega_{\beta_0})$ while $x \mapsto \sigma(x)$ is in $C^1(\Omega_{\beta_0})$.
\vskip 2mm

In addition to \eqref{A1} and \eqref{A2}, we assume that the ground state $\Phi_V$ satisfies the following condition: 

There exist $a_0\geq 1$ and $\ga,\ga^*>0$ satisfying
$$
0 \leq \alpha - \alpha^* 	< \frac{1}{2},
$$
\sth for every $a>a_0$ and every $x,z \in \Omega_{\beta_0}$ lying on a normal to $\partial \Omega$:
\begin{equation} \tag{C1} \label{C1}
	a\delta(x) \leq \delta(z) \Longrightarrow \frac{\Phi_V(x)}{\Phi_V(z)} \leq c(a)\frac{\delta(x)^{\alpha^*}}{\delta(z)^\alpha}.
\end{equation}

Conditions \eqref{A1}, \eqref{A2} and \eqref{C1} are assumed, without further mention,  throughout the paper.

\begin{definition}\label{d:sol} 
	Let $(\tau, \nu)\in \GTM(\Gw;\Phi_V) \ti \GTM(\bdw)$ and $u\in L^1_{\rm loc}(\Gw)$.
		
(i) $u$ is a solution of  \eqref{bvp-gen} if  $f(u)\in L^1(\Gw;\Phi_V)$, the equation holds in the distribution sense and $\tr_V u=\nu$.

(ii) $u$ is a subsolution of  \eqref{bvp-gen} if  $f(u)\in L^1(\Gw;\Phi_V)$, $-L_Vu +f(u)\leq \tau$ in the distribution sense and $\tr_V u \leq \nu$.

A supersolution is defined in the same way with the inverse inequalities.
\end{definition}

With this definition, $u$ is a solution of \eqref{bvp-gen} if and only if (see \cite[Lemma 3.1]{MM-note})
\begin{equation}\label{bvp-GK}
u+\BBG_V[f(u)]=\BBG_V[\tau] +\BBK_V[\nu] \quad \text{in } \Omega.
\end{equation}

 If \eqref{bvp-gen} has a solution we say that $(\tau,\nu)$ is a \emph{good couple}.

If $(0,\nu)$ (respectively $(\tau,0)$) is a good couple we say that $\nu$ (respectively $\tau$) is a \emph{good measure}.

We are interested in questions related to the notion of `reduced measure' introduced in \cite{BMP} (for $L_V=\Gd$). In general, problem \eqref{bvp-gen} is not solvable for every couple $(\tau,\nu)$. 

A great deal of research has been devoted to a precise characterization of good measures or good couples in some specific cases. Most of this research dealt with the equation $-\Gd u +f(u)=0$ in $\Omega$ and in particular with the case $f(t)=|t|^p\sign t$, $p>1$ (see  \cite{Dybook2,LG95,LG99,Ms2004} for $1<p\leq 2$ and \cite{MV1, MV2, MM} for every $p>1$ and the
references therein). 
See also \cite{An-Mar} where the problem was treated for a general class of nonlinearities $f$ that satisfy the Keller - Osserman condition.

More recently the characterization of good measures was studied \wrto the equation $-L_Vu +f(u)=0$ in $\Omega$, mainly when $V$ is the Hardy potential and $f(t)=|t|^p\sign t$  (see, e.g. \cite{MN1,MN2,M+Moroz,Gki-Ver,CV}).  The question was also studied in the context of fractional Schr\"odinger equations  (see, e.g. \cite{Go-Vaz}).

The idea of `reduced measure', introduced in \cite{BMP},
is to provide a reduction process that converges to the  `good' part of $\tau$ and $\nu$  when \eqref{bvp-gen} has no solution.\footnote{A related notion of 'reduced limit' was studied in \cite{MP}, \cite{M+Bhakta}.} 

In \cite{BMP}, the authors study problem \eqref{bvp-gen} for $L_V=\Gd$, mainly in the case where $\nu=0$, $\tau\in \GTM(\Gw)$, assuming that $f$ satisfies \eqref{f-cond} and vanishes on $(-\infty, 0]$. A solution is a function $u\in L^1(\Gw)$ \sth $f(u)\in L^1(\Gw)$ and $u$ satisfies \eqref{bvp-gen} in the weak sense. To determine the reduced measure the authors consider a sequence of problems
\begin{equation}\label{bvp-gdn}
	-\Gd u + f_n(u)=\tau \qtxt{in }\Gw, \q u=0 \qtxt{on }\bdw,
\end{equation}
where $f_n\in C(\BBR)$ is a non-negative, nondecreasing function, $f_n\uparrow f$ and \eqref{bvp-gdn} has a solution for every $\tau\in \GTM(\Gw)$. (For instance, the functions $f_n$ are bounded.) One of the main results 
states (see \cite[Theorem 4.1]{BMP}):

Let  $u_n$ be the unique solution of \eqref{bvp-gdn}. Then the sequence $\{u_n\}$
decreases and  $u^*:=\lim u_n$ satisfies
\begin{equation*}
	-\Gd u^* + f(u^*)=\tau^* \qtxt{in }\Gw, \q u^*=0 \qtxt{on }\bdw,
\end{equation*}
where $\tau^*$ is the largest good measure dominated by $\tau$. 

In \cite{BP} a similar result is established in the case $(\tau,\nu)\in \GTM(\Gw)\ti \GTM(\bdw)$ (possibly signed measures) assuming as before that $f=0$ on $(-\infty,0]$.
It is also shown that $(\tau,\nu)$ is good if and only if, both
$\tau$ and $\nu$ are good measures.

When $\tau, \nu$ are signed measures and we drop the assumption `$f=0$ on $(-\infty,0)$', the situation is more complex even in the case $V=0$.

The present definition of a solution of \eqref{bvp-gen} is necessarily different from that of weak solution used in \cite{BMP}, \cite{BP}. But when $V=0$ these are essentially equivalent.

In the case where $\tau$ and $\nu$ are positive our results are similar to those quoted above. 

Let $u_n$ denote the solution of the problem,
\begin{equation}\label{bvp-gfn}
	-L_V u + f_n(u)=\tau \qtxt{in }\Gw, \q \tr_V u=\nu \qtxt{on }\bdw 
\end{equation}
where $f_n\in C(\BBR)$ is non-decreasing, bounded and $(f_n)_\pm\uparrow f_\pm$.

\begin{theorem} \label{th:main1}
	Assume that $f$ satisfies \eqref{f-cond} and 
	$(\tau,\nu)\in \GTM_+(\Gw;\Phi_V)\ti \GTM_+(\bdw)$. For $u_n$ as above: $u_n\downarrow u^\#$ and
	\begin{equation}\label{bvp-pos}
-L_V u^\# + f(u^\#) = \tau^\# \quad \text{in } \Omega, \q \tr_Vu^\#=\nu^\#	
\end{equation}
where $(0,0) \prec (\tau^\#,\nu^\#) \prec (\tau,\nu)$.
Moreover, $u^\#$ is the largest subsolution of \eqref{bvp-gen} and  $\tau^\#$, $\nu^\#$ are the largest good measures dominated by $\tau$ and $\nu$ respectively.

\end{theorem}

A corresponding result holds for couples of negative measures (see Remark \ref{neg-m}).

In the case that $\tau, \nu$ may be signed measures we prove:


\begin{theorem}\label{th:main2}
	Assume that $f$ satisfies \eqref{f-cond} and
$(\tau,\nu)\in \GTM(\Gw;\Phi_V)\ti \GTM(\bdw)$. Let $(\gl, \gs)$ be a couple of measures \sth 
	\begin{equation}\label{interval1}
-(\tau_-, \nu_-)\prec (\gl,\gs)\prec (\tau_+,\nu_+)
\end{equation}
and let $u_n$ be the solution of \eqref{bvp-gfn} with $(\tau,\nu)$ replaced by $(\gl,\gs)$. 

Every subsequence of $\{u_n\}$ has a limit point \wrto a.e. convergence.
   If $\tl u$ is such a limit point then 
\begin{equation}\label{bvp_tl}
-L_V \tl u + f(\tl u) = \tl \gl \qtxt{in }\;\Gw, \q 
\tr_V\tl u=\tl \gs 
\end{equation}
and
	\begin{equation}
	(-\tau_-, -\nu_-)^\# \prec (\tl\gl,\tl\gs) \prec (\tau_+,\nu_+)^\#.
\end{equation} 

Moreover, every couple $(\gl,\gs)$ \sth 
\begin{equation}\label{good_interval}
(-\tau_-, -\nu_-)^\# \prec (\gl,\gs) \prec (\tau_+,\nu_+)^\#
\end{equation}
is a good couple.
\end{theorem}

This naturally leads to the following  question: 
\emph{If $(\gl,\gs)$ is a good couple in the interval \eqref{interval1} 
	does it necessarily satisfy \eqref{good_interval}?}

As shown below, if $f$ vanishes in $(-\infty,0]$, the answer is positive. In the general case this is an open question.

\begin{theorem} \label{th:main3}
	Assume that $f$ satisfies \eqref{f-cond} and that $f(t)=0$ for $t\leq 0$.
	Let $(\tau,\nu)\in \GTM(\Gw;\Phi_V)\ti \GTM(\bdw)$. Then $(\tau,\nu)$ is a good couple if and only if
	\begin{equation}\label{good_couple}
		(\tau,\nu) \prec (\tau_+^\#,\nu_+^\#)
	\end{equation}
	where $\tau_+^\#$ and $\nu_+^\#$ are the largest good measures dominated by $\tau_+$ and $\nu_+$ respectively.\footnote{Note that, when $f=0$ on $(-\infty,0]$, the couple $(-\tau_-,- \nu_-)$ is always good.}  
	
	Consequently, $(\tau,\nu)$ is a good couple with respect to \eqref{bvp-gen} if and only if $\tau$ and $\nu$ separately are good measures.
\end{theorem}

\noindent This result extends \cite[Theorem 6]{BP}. 

\vskip 2mm


Our main tools include: two-sided estimates of $\BBG_V[\tau]$, $\tau\in \GTM_+(\Gw;\Phi_V)$, and $\BBK_V[\nu]$, $\nu\in \GTM_+(\bdw)$ \cite{MM-22}, the inverse maximum principle \cite{Dup-Pon}, Kato's inequality and its extension  due to \cite{BP'}  and a result of \cite{BMP} 
regarding the diffuse part of reduced measures.

Here is the plan of the paper. In Section 2 we recall the estimates of \cite{MM-22} as well as properties of sub and supersolutions established in \cite{MM-note} that are frequently used in the present paper.
In Section 3 we study the problem of reduced couple for \eqref{bvp-gen} with positive measures. Theorem \ref{th:main1} is a consequence of Theorems \ref{t:reduced2} and \ref{t:reduced3}. Section 4 is devoted to problem \eqref{bvp-gen} with signed measures. Theorem \ref{th:main2} is a consequence of Theorem \ref{signed-reduced} and Proposition \ref{g.couple}. The section is completed by the proof of Theorem \ref{th:main3}.

\section{Some results on sub and supersolutions.}

In this section we gather several results from \cite{MM-22} and \cite{MM-note} that are frequently used  in the sequel.

\subsection{Estimates of $L_V$ harmonic functions and $L_V$ potentials.}
$\,$\vskip 1mm

\noindent The estimates stated below are derived in \cite{MM-22}.

\begin{theorem} \label{Knu} (\cite[Theorem 3.1]{MM-22})
Assume that \eqref{A1}, \eqref{A2} and \eqref{C1} hold. Then for any $\nu\in \GTM_+(\partial \Omega)$, 
$$
	\rec{C}\|\nu\|_{\GTM(\partial \Omega)}\leq\int_{\Gs_\gb}\frac{\Phi_V}{\delta}\BBK_V[\nu]\,dS \leq C\|\nu\|_{\GTM(\partial \Omega)} \quad \forall \beta \in (0,\beta_0),
$$
	where the constant $C$ depends on $\bar a,\Gw$ and the constants in \eqref{C1}.
\end{theorem}

Next is an estimate of $L_V$ potentials.

\begin{theorem} \label{G-I} (\cite[Theorems 3.3 and 3.4]{MM-22})
	
	(i) Assume \eqref{A1} and \eqref{A2} hold.
	Then there exists a  constant $c$ depending on $\bar a$, $\Omega$ such that, for every $\tau\in \GTM_+(\Gw;\Phi_V)$,
	$$
	\rec{c}\int_{\Gw} \Phi_V \,d\tau\leq \int_{\Gw}\frac{\Phi_V}{\delta}\BBG_V[\tau]\,dx.
	$$

(ii)   Assume \eqref{A1}, \eqref{A2} and \eqref{C1} hold. Then there exists $c'>0$ depending on $\bar a$, $\Gw$ and the constants in \eqref{C1} such that for every $\tau\in \GTM_+(\Gw;\Phi_V)$,
$$
\int_{\Gw}\frac{\Phi_V}{\delta}\BBG_V[\tau] \,dx \leq c'\int_\Gw \Phi_V \,d\tau.
$$
\end{theorem}

\subsection{Remarks on subsolutions and supersolutions}

We list some properties of subsolutions and supersolutions from \cite{MM-note}.

\begin{lemma}\label{subhar=0} 
Let $w \in L_{\rm loc}^1(\Omega)$ be a non-negative $L_V$ subharmonic function. If $\tr_V w=0$ then $w\equiv 0$. 
\end{lemma}

This is a consequence of \cite[Corollary 2.6]{MM-note}.

\begin{lemma}\label{exist-KV} (\cite[Corollary 2.8]{MM-note}) 
Let $u \in L_{\rm loc}^1(\Omega)$ and suppose that $-L_Vu=\tau\in \GTM(\Gw;\Phi_V)$.
In addition assume that for some smooth exhaustion $\{D_n\}$ of $\Omega$, 
\begin{equation} \label{supDn} \sup_{n} \int_{\partial D_n} P_{V,n}(x_0,y)|u(y)|\, dS(y) < + \infty.
\end{equation}
Then $\tr_V u=:\nu$ exists and $u= \BBG_V[\tau]+\BBK_V[\nu]$ in $\Omega$.
\end{lemma}

\Remark  If $\tr_Vu$ exists then, by definition, \eqref{supDn} holds.

If $u\geq 0$,  condition  \eqref{supDn} is not needed. In this case the result (stated below) is a consequence of the Riesz decomposition lemma.

\begin{lemma}\label{exist_tr_pos} (\cite[Lemma 2.11]{MM-note}) 
Let $u\in L^1_{\rm loc}(\Gw)$ be a positive function such that $-L_Vu=\tau\in \GTM(\Gw;\Phi_V)$. Then $u$ has an $L_V$ \btr, say $\nu$, and $u=\BBG_V[\tau] +\BBK_V[\nu]$ in $\Omega$.

\end{lemma}
\noindent\textit{Notation.}\hskip 2mm   A function $u$ is \textit{$L_V$ perfect} if $u=\BBG_V[\tau] + \BBK_V[\nu]$ for some $\tau \in \GTM(\Omega;\Phi_V)$ and $\nu \in \GTM(\partial \Omega)$.

\begin{lemma} \label{tr+} (\cite[Lemma 3.3]{MM-note})
Suppose that $u$ is an 	$L_V$ perfect function. If $\tr_V u \leq 0$ then $\tr_V u_+ = 0$.
\end{lemma}

Consider the equation
\begin{equation}\label{Eq=tau}
-L_V u+f(u)= \tau \quad \text{in } \Omega
\end{equation}
and the \bvp
\begin{equation}\label{bvp-f}
-L_V u+f(u)=\tau \qtxt {in } \Gw, \q \tr_V u=\nu,
\end{equation}
where $f$ satisfies \eqref{f-cond}, $\tau\in \GTM(\Gw;\Phi_V)$ and $\nu \in \GTM(\partial \Omega)$.

\begin{definition}\label{d:sub-sup}
	 Let $u\in L^1_{\rm loc}(\Gw)$ be a function \sth $f(u)\in L^1_{\rm loc}(\Gw)$.
	
	The function $u$ is a subsolution (supersolution) of \eqref{Eq=tau} if 
	$-L_Vu + f(u)\leq (\geq)\tau$ in $\Omega$ in the distribution sense.

	The function $u$ is a subsolution (supersolution) of \eqref{bvp-f} if it is a subsolution (supersolution) of \eqref{Eq=tau},
	$f(u)\in L^1(\Gw;\Phi_V)$ and $u$ has an $L_V$ boundary trace such that $\tr_Vu\leq \nu$ ($\tr_Vu\geq \nu$).
\end{definition}

\begin{lemma}\label{l:subsup_rep} (\cite[Lemma 3.1]{MM-note})
	Assume that $u\in L^1_{\rm loc}(\Gw)$ and $f(u)\in L^1(\Gw;\Phi_V)$.
	
	If $u$ is a subsolution (resp. supersolution) of problem \eqref{bvp-f} then $u$ is $L_V$ perfect.
	More precisely, there exist measures $\gl\in \GTM_+(\Gw;\Phi_V)$ and $\gs\in \GTM_+(\bdw)$ \sth:
\begin{equation}\label{e:subsup_rep}\BAL
	u &=  \BBG_V[\tau-f(u) -\gl]+\BBK_V[\nu-\gs] \\ (u &=  \BBG_V[\tau-f(u) +\gl]+\BBK_V[\nu+\gs]). 
	\EAL\end{equation}
\end{lemma}

\begin{lemma} \label{sub-sup} (\cite[Lemma 3.4]{MM-note})
(i)	Let $u_1$ (resp. $u_2$)  be a supersolution (resp. subsolution) of \eqref{bvp-f}. Then $u_2\leq u_1$.

(ii) Problem \eqref{bvp-f} has at most one solution.
\end{lemma}

%

%

\begin{lemma}\label{exist-subsup} (\cite[Corollary 3.7]{MM-note}) 
Suppose that $u_1$, $u_2$ are respectively a supersolution and a subsolution of \eqref{Eq=tau} \sth $f(u_i)\in L^1(\Gw;\Phi_V)$. If $u_i$ has an $L_V$ \btr, say $\nu_i$, $i=1,2$, and $\nu_2 \leq  \nu_1$ then problem \eqref{bvp-f} has a (unique) solution for every measure $\nu$ \sth $\nu_2\leq\nu\leq \nu_1$.
\end{lemma}

\section{The reduced measures for couples of positive measures.}
\noindent\textit{Notation.}\hskip 2mm If $f$ and $g$ are two non-negative functions on a set $X$,
we say that $f$ and $g$ are similar if there exists $c>0$ \sth
$$\rec{c}f(x)\leq g(x)\leq cf(x) \forevery x\in X.$$
This relation is denoted by $f\sim g$. 

\begin{theorem}\label{t:reduced2}
Let $(\tau,\nu)\in \GTM_+(\Gw;\Phi_V) \ti \GTM_+(\bdw)$. Let $\{f_n\}$ be a sequence of continuous, bounded, non-decreasing functions on $\BBR$ \sth $f_n(0)=0$ and $(f_n)_\pm\uparrow f_\pm$.

 Then the \bvp 
\begin{equation}\label{bvpn'}
-L_Vu +f_n(u) =\tau \qtxt{in }\Gw, \q \tr_V u=\nu,
\end{equation}
has a unique solution  $u_n=u_n(\tau,\nu)$.
The sequence $\{u_n\}$ is a decreasing sequence of positive functions and its limit $u^\#= u^\#(\tau,\nu)$ has the following properties.

(a) $u^\#$ satisfies
\begin{equation}\label{u*1}
\norm{u^\#}_{L^1(\Gw;\Phi_V/\gd)} +\norm{f(u^\#)}_{L^1(\Gw;\Phi_V)}\leq C (\norm{\nu}_{\GTM(\partial \Omega)} + \norm{\tau}_{\GTM(\Gw;\Phi_V)})
\end{equation}
and
\begin{equation}\label{ineq-u*}
u^\#+ \BBG_V[f(u^\#)] \leq \BBG_V[\tau]+\BBK_V[\nu]=:\tl w \quad \text{in } \Omega.
\end{equation}

(b) There exists a non-negative measure $\tau^\#\leq \tau$ \sth
\begin{equation}\label{equ*}
-L_V u^\#+ f(u^\#)=\tau^\# \quad \text{in } \Omega.
\end{equation}

(c) $u^\#$ has $L_V$ boundary trace $0\leq \nu^\#\leq \nu$. Thus
\begin{equation}\label{u-sol}
u^\#+\BBG_V[f(u^\#)]= \BBG_V[\tau^\#] + \BBK_V[\nu^\#] \quad \text{in } \Omega.
\end{equation} 

(d) $u^\#$ is the largest subsolution of problem 
\begin{equation}\label{bvp-tau}
-L_V u + f(u)=\tau \quad \text{in } \Omega, \q \tr_V u=\nu.
\end{equation}
In particular, if \eqref{bvp-tau} has a solution $u$ then $u^\#=u$,  $\tau^\#=\tau$ and $\nu^\#=\nu$.

\end{theorem}

\begin{proof} 
 The function $\tl w := \BBK_V[\nu] + \BBG_V[\tau]$ is a supersolution of the equation
 $$-L_Vu+f_n(u)=\tau \quad \text{in } \Omega$$ 
and $\tr_V\tl w=\nu$. Obviously $v \equiv 0$ is a subsolution, $v\leq \tl w$ and $f_n(\tl w)\in L^1(\Gw;\Phi_V)$. 

Therefore, by Lemma \ref{exist-subsup}, there exists a unique solution $u_n = u_n(\tau,\nu)$ of the \bvp
\eqref{bvpn'} satisfying $0 \leq u_n \leq \tilde w$ in $\Omega$. The solution $u_n$ satisfies
\begin{equation}\label{eq-un}
u_n+\BBG_V[f_n(u_n)]= \BBG_V[\tau]+\BBK_V[\nu] = \tl w \quad \text{in } \Omega.
\end{equation}
Put $w:=u_{n+1}-u_n$. Then $w$ is $L_V$ perfect and   $\tr_V w =0$. Consequently, by Lemma \ref{tr+}, $\tr_V w_+=0$. Furthermore
$$-L_V w + f_{n+1}(u_{n+1})-f_n(u_n) =0 \quad \text{in } \Omega.$$
 By Kato's inequality
$$-L_V w_+ +  (f_{n+1}(u_{n+1})-f_n(u_n))\sign_+ w \leq 0 \quad \text{in } \Omega.$$ 
In the set $\{x\in\Gw: w_+\geq 0\}$: 
$$f_{n+1}(u_{n+1})-f_n(u_n) \geq f_{n+1}(u_n)-f_n(u_n) \geq 0.$$ Thus $w_+$ is  \LVsub in $\Omega$. As $\tr_Vw_+=0$, Lemma \ref{subhar=0} yields  $w_+=0$, i.e., $u_{n+1}\leq u_n$ in $\Omega$.

(a) \hskip 2mm Let $u^\# = \lim u_n$, then $0 \leq u^\# \leq \tilde w$ in $\Omega$. By Dini's Lemma, $f_n(u_n)\to f(u^\#)$ a.e. in $\Omega$.
Therefore, by \eqref{eq-un} and
Fatou's lemma, we obtain \eqref{ineq-u*}.

By Theorems \ref{Knu} and \ref{G-I}, we have
$$\int_\Gw \frac{\Phi_V}{\gd}\tl w\, dx = \int_\Gw \frac{\Phi_V}{\gd} (\BBK_V[\nu] + \BBG_V[\tau]) \,dx \sim \norm{\tau}_{\GTM(\Gw;\Phi_V)}+\norm{\nu}_{\GTM(\partial \Omega)} $$
and
$$\int_\Gw \frac{\Phi_V}{\gd}\BBG_V[f_n(u_n)]\,dx \sim \int_\Gw f_n(u_n) \Phi_V \,dx.$$
Therefore, multiplying \eqref{eq-un} by $\Phi_V/\delta$ and integrating over $\Gw$, we obtain the following similarity relations
\begin{equation}\label{sim-un} \begin{aligned}
&\int_\Gw u_n \frac{\Phi_V}{\gd}\,dx +\int_\Gw f_n(u_n)\Phi_V\,dx \\
 &\sim \int_\Gw u_n \frac{\Phi_V}{\gd}\,dx + \int_\Gw \frac{\Phi_V}{\gd}\BBG_V[f_n(u_n)]\,dx \\
 &= \int_\Gw \tl w \frac{\Phi_V}{\gd}\,dx \sim \norm{\tau}_{\GTM(\Gw;\Phi_V)}+\norm{\nu}_{\GTM(\bdw)}.
\end{aligned}
\end{equation}
Letting $n\to \infty$ and using Fatou's lemma, we obtain \eqref{u*1}.

(b) Let $\gz\in C_c^\infty(\Gw)$. Multiplying \eqref{eq-un} by $-L_V\gz$ and integrating over $\Omega$, we obtain
\begin{equation*}
- \int_\Gw u_n L_V\gz\, dx - \int_\Gw \BBG_V[f_n(u_n)]L_V\gz\, dx = - \int_\Gw \BBG_V[\tau] L_V\gz\, dx.
\end{equation*}
We used the fact that $\BBK_V[\nu]$ is $L_V$ harmonic. Further, for every $\tau\in \GTM(\Gw;\Phi_V)$,  $-L_V \BBG_V[\tau] = \tau$, i.e.,
$$ -\int_\Gw  \BBG_V[\tau]L_V\gz\, dx = \int_\Gw \gz \, d\tau.$$
Hence
\begin{equation*}
- \int_\Gw u_n L_V\gz\, dx +\int_\Gw f_n(u_n)\gz\, dx = \int_\Gw \gz \, d\tau.
\end{equation*}

Recall that $\{u_n\}$ converges to $u^\#$ and is dominated by $u_1$  in $L^1(\Gw;\Phi_V/\delta)$ and $\{ f_n(u_n)\}$ converges to $f(u^\#)$ and -- by \eqref{sim-un} -- is bounded in $L^1(\Gw;\Phi_V)$. 
Consequently, using Fatou's lemma, 
$$
-C\int_\Gw \gz\,dx\leq - \int_\Gw u^\# L_V\gz\, dx + \int_\Gw f(u^\#)\gz\, dx \leq \int_\Gw \gz \, d\tau
$$
for every $0\leq \gz\in C_c^\infty(\Gw)$ where $C\geq 0$ a constant independent of $\gz$.
Thus there exists a bounded measure $\tau^\#\leq \tau$ \sth \eqref{equ*} holds. 

Now, in  
Lemma \ref{tau*d} below, it is shown that $\tau^\#\geq \tau_d$. This is based only on the assumptions of the present theorem, the definition of $u^\#$ as the limit of the 
decreasing sequence $\{u_n\}$ and what is proved above. Therefore, as $\tau\geq 0$, we conclude that $\tau^\#\geq 0$. 
This completes the proof of part (b).

(c) By \eqref{ineq-u*}, $u^\#\leq \tl w$. Since $\tr_V\tl w$ exists and $u^\# \geq 0$, it follows that 
	$$\BAL
&\sup_{n} \int_{\partial D_n} P_{V,n}(x_0,y)|u^\#(y)|dS(y) \leq \\
&\sup_{n} \int_{\partial D_n} P_{V,n}(x_0,y) \tilde w(y) dS(y) < +\infty.
\EAL$$
By \eqref{equ*}, as $f(u^\#)\in L^1(\Gw;\Phi_V)$, the function $v:=u^\#+\BBG_V[f(u^\#)]$ satisfies $-L_V v=\tau$ and $v\geq 0$. Therefore, by Lemma \ref{exist-KV} and \eqref{tr-prop'}, $\tr_V\,v=:\nu^\#$ exists, $\tr_V\,u^\# = \tr_V\,v$ and \eqref {u-sol} holds.
As $0\leq u^\#\leq \tl w$, 
$$0\leq \nu^\#\leq \tr_V\tl w=\nu.$$ 

(d) Let $w$ be a positive subsolution of \eqref{bvp-tau}. Then
$$-L_Vw + f_n(w)\leq -L_Vw + f(w) \leq \tau \quad \text{in } \Omega,\q \tr_V w\leq \nu.$$
On the other hand, we have
$$-L_Vu_n + f_n(u_n)= \tau \quad \text{in } \Omega,\q \tr_V u_n = \nu.$$
By Lemma \ref{sub-sup}, $w\leq u_n$ and thus $w\leq u^\#$.
This proves (d). 

Obviously, if \eqref{bvp-tau} has a solution $u$ then it is the largest subsolution of the problem so that  $u^\#=u$.
\end{proof}

\begin{definition}\label{reduced_couple}
A measure $\tau \in \GTM_+(\Gw;\Phi_V)$ is a \textit{good measure} \wrto $f$ if there exists a solution $u$ of equation \eqref{Eq=tau} \sth $f(u)\in L^1(\Gw;\Phi_V)$.

A couple of measures $(\tau,\nu)\in \GTM_+(\Gw;\Phi_V)\ti \GTM_+(\bdw)$ is a \textit{good couple} \wrto $f$ if there exists a solution $u$ of problem \eqref{bvp-tau}.

The couple $(\tau^\#,\nu^\#)$ that satisfies \eqref{u-sol}  is called the \textit{reduced couple} of $(\tau,\nu)$.
\end{definition}

\remark\label{exp1} We note that as  a consequence of Theorem \ref {t:reduced2} parts (b) and (c):\\ [1mm]
(i) For every $\nu\in \GTM_+(\bdw)$ the reduced couple of $(0,\nu)$ is $(0,\nu^*)$ where $\nu^*$ is the largest good measure dominated by $\nu$.\\
(ii) For every $\tau\in \GTM_+(\Gw;\Phi_V)$ the reduced couple of $(\tau,0)$ is $(\tau^*,0)$  where $\tau^*$ is the largest good measure dominated by $\tau$.

\vskip 2mm

\noindent\textit{Notation.}\hskip 2mm (a)  Let $\tau\in \GTM_+(\Gw;\Phi_V)$. Denote by $\tau^\#(\nu)$ the measure $\tau^\#$ in Theorem \ref{t:reduced2}. In particular $\tau^\#(0)$
  is the reduced measure of problem
$$
L_Vu + f(u)=\tau \quad \text{in } \Omega, \q \tr_V u=0.
$$

(b) Let $\gl$ be a Borel measure in $\Gw$ \sth $\gl=\gl_+-\gl_-$ where $\gl_\pm$ are positive Radon measures. It is well-known (see e.g. \cite{BMP}) that $\gl$ has a unique representation of the form $\gl= \gl_c+\gl_d$ where $\gl_d$ vanishes on sets of (Newtonian) capacity zero while  $\gl_c$ is concentrated on a set of zero capacity. We say that $\gl_d$ is the \emph{diffuse} part of $\gl$ while $\gl_c$ is the \textit{concentrated} part of $\gl$. If $\gl=\gl_d$ we say that $\gl$ is a diffuse  measure. If  $\gl=\gl_c$ we say that $\gl$ is a concentrated measure.

For the proof of the next theorem we need a version of \cite[Lemma 4.1]{BMP} suitable for the present problem. The proof is essentially the same as in \cite{BMP}, but some slight modifications are needed. 
For the convenience of the reader we provide the proof below.

\begin{lemma}\label{tau*d}
Let $(\tau,\nu)\in \GTM_+(\Gw;\Phi_V) \ti \GTM_+(\bdw)$. Then under the assumptions and with the notation of Theorem \ref{t:reduced2}, we have
\begin{equation}\label{taud}
 \tau^\#\geq \tau_d \qtxt{and}\q   (\tau^\#)_d=\tau_d.
\end{equation} 
\end{lemma}

\begin{proof}  Let $u_n$ be the solution of \eqref{bvpn'}. Then $u_n\geq 0$, the sequence $\{u_n\}$ is decreasing and, as in Theorem \ref{t:reduced2}, the function $u^\# :=\lim u_n$ satisfies 
\begin{equation}\label{taud0}
-L_V u^\# +f(u^\#)=\tau^\# \quad \text{in } \Omega,
\end{equation}
where $\tau^\#$ is a measure \sth $\tau^\# \leq \tau$.

Denote $T_k(s):= \min(k,s)$, $s\in \BBR$. By \cite{BP'} and \cite[Corollary 4.9]{BMP},
$$ \Gd T_k(u_n)\leq \chi_{[u_n\leq k]}(\Gd u_n)_d + ((\Gd u_n)_c)_+,$$
where $\chi_A$ denotes the characteristic function of $A \subset \R^N$. 
 Since $u_n$  satisfies \eqref{bvpn'}, we obtain
$$ (\Gd u_n)_d = -Vu_n + f_n(u_n) -\tau_d, \q (\Gd u_n)_c=-\tau_c.$$
As $\tau\geq 0$, these relations and the previous inequality yield
$$ \Gd T_k(u_n)\leq \chi_{[u_n\leq k]} ( -Vu_n + f_n(u_n) -\tau_d) \quad \text{in } \Omega,$$
and
$$\BAL -L_V T_k(u_n) &\geq -VT_k(u_n) +  \chi_{[u_n\leq k]}(Vu_n - f_n(u_n) +\tau_d) \\
&= -\chi_{[u_n> k]}V  k  -  \chi_{[u_n\leq k]} (f_n(u_n) -\tau_d).
\EAL $$
Since $\chi_{[u_n\leq k]}f_n(u_n) \leq f_n(T_k(u_n))$, it follows that
\begin{equation}\label{taud1}
-L_V T_k(u_n)+  f_n(T_k(u_n))\geq  -\chi_{[u_n> k]}|V| u_n +  \chi_{[u_n\leq k]} \tau_d \quad \text{in } \Omega.
\end{equation}
Let $\gz\in C_c^\infty(\Gw)$ and $\gz\geq 0$. Since $0\leq u_n\leq u_1$,  \eqref{taud1} yields
\begin{equation*}\label{taud2}\BAL
&-\int_\Gw T_k(u_n)L_V\gz\,dx + \int_\Gw f_n(T_k(u_n))\gz \,dx\geq \\  
&-\int_\Gw \chi_{[u_1> k]}|V| u_1\gz\,dx  + \int_\Gw \chi_{[u_1\leq k]}\gz\, d\tau_d.
\EAL\end{equation*} 
By the dominated convergence theorem, letting $n\to\infty$, we obtain
\begin{equation*}\label{taud3}\BAL
&-\int_\Gw T_k(u^\#)L_V\gz\,dx + \int_\Gw f(T_k(u^\#))\gz \,dx\geq \\ 
&-\int_\Gw \chi_{[u_1> k]}|V| u_1\gz\,dx  +  \int_\Gw \chi_{[u_1\leq k]}\gz\, d\tau_d.
\EAL\end{equation*}
Finally, letting $k\to \infty$, we obtain
$$
-\int_\Gw u^\# L_V\gz\,dx + \int_\Gw f(u^\#)\gz \,dx\geq \int_\Gw \gz\, d\tau_d.
$$
In view of \eqref{taud0} this implies $\tau^\#\geq \tau_d$ and therefore $(\tau^\#)_d\geq \tau_d$. As $\tau^\#\leq \tau$ we obtain $(\tau^\#)_d = \tau_d$. This proves
\eqref{taud}.
\end{proof}

%

\begin{theorem}\label{t:reduced3}
Under the assumptions and with the notation of Theorem \ref{t:reduced2}, the following statements hold.

(i) For every $\tau$, $\nu^\#=\nu^*$ with $\nu^*$ as in Remark \ref{exp1}. Thus $\nu^\#$ does not depend on the data $\tau$.

(ii) For every $\nu$, $\tau^\#=\tau^*$ with $\tau^*$ as in Remark \ref{exp1}. Thus $\tau^\#$ does not depend on the boundary data $\nu$.

(iii) 
Let $(0,0) \prec (\gl,\sigma) \prec (\tau,\nu)$. Then problem
$$
-L_Vu +f(u)=\gl \quad \text{in } \Omega, \q \tr_V u=\gs,
$$
 has a solution if and only if $(\gl,\gs) \prec  (\tau^*,\nu^*)$.
\end{theorem}

\begin{proof} (i) Let $u_n(\tau,\nu)$, $u^\#(\tau,\nu)$ and $(\tau^\#, \nu^\#)$ be as in Theorem \ref{t:reduced2}. For simplicity, we write $u^*=u^\#(0,\nu)$ and $u^\# = u^\#(\tau,\nu)$. Since $0 \leq u_n(0,\nu)\leq u_n(\tau,\nu)$, it follows that $0 \leq u^* \leq u^\#$ in $\Omega$.  Therefore, using Theorem \ref{t:reduced2}(c), we obtain
\begin{equation}\label{tr-ineq1}
\nu^*= \tr_V u^* \leq \tr_V u^\# =\nu^\#\leq \nu.
\end{equation} 
Consequently $u^*$ is a subsolution of problem
\begin{equation} \label{eq0nu}
- L_V u + f(u) = 0 \quad \text{in } \Omega, \quad \tr_V u = \nu^\#.	
\end{equation}
Obviously $u^\#$ is a supersolution of \eqref{eq0nu}.
Hence, by Lemma \ref{exist-subsup}, there exists a unique solution $\bar v$ of problem \eqref{eq0nu} and $0 \leq u^*\leq \bar v \leq u^\#$. By Theorem \ref{t:reduced2}, $u^*$ is the largest solution of the equation $-L_Vu+f(u)=0$ in $\Omega$ with $L_V$ boundary trace $\leq \nu$. Thus $\bar v \leq u^*$ and consequently  $\nu^\#=\tr_V \bar v \leq \tr_V u^* = \nu^*$. This inequality and \eqref{tr-ineq1} imply $\nu^\#=\nu^*$.
\vskip 2mm

(ii) Here we denote by $R(\tau,\nu)$ the reduced couple of $(\tau, \nu)$. For given $\tau$ we denote by
$\tau^\#(\nu)$ the first component of $R(\tau,\nu)$.

In view of Theorem \ref {t:reduced2}(d), 
 $$R(\tau, \nu^\#)= R(\tau,\nu),\q \tau^\#(\nu)=\tau^\#(\nu^\#).$$
Hence, as $\nu^\#=\nu^*$,  
\begin{equation}\label{nu*1}
R(\tau,\nu)=R(\tau, \nu^*), \q u^\#(\tau,\nu)=u^\#(\tau, \nu^*).
\end{equation} 
Therefore, in this part of the proof we may assume that $\nu=\nu^*$.



\noindent\textbf{Step I.}\hskip 2mm If $\nu_1, \nu_2\in \GTM_+(\bdw)$ and $\nu_1\leq \nu_2$ then
\begin{equation}\label{nu1,2}
\tau^\#(\nu_1) \geq \tau^\#(\nu_2).
\end{equation}
\proof The assumption implies that $\nu_1^*\leq \nu_2^*$. Therefore by \eqref{nu*1}, we may assume that $\nu_1, \nu_2$ are good measures \wrto \eqref{bvp-tau}.

 Let $\gl$ be a measure \sth $0\leq\gl\leq\tau$ and suppose that there exists a solution $u_2$ of
$$ -Lu+f(u)=\gl \quad \text{in } \Omega,  \q    \tr_V u=\nu_2. 
$$
Then $u_2$ is a supersolution of
\begin{equation}  \label{a-2}    -Lu+f(u)=\gl \quad \text{in } \Omega,  \q    \tr_V u=\nu_1. 
\end{equation}

We also know that there exists $\gl^\#\leq \gl$ \sth the following problem has a solution $u_1$: 
$$         -Lu+f(u)=\gl^\# \quad \text{in } \Omega,  \q    \tr_V u=\nu_1. 
$$
The function $u_1$ is a subsolution of \eqref{a-2} and, by Lemma \ref{sub-sup}, $u_1 \leq  u_2.$ 
Hence there exists a solution $\bar v$ of \eqref{a-2}. Clearly, $\bar v$ is a subsolution of 
$$-L_V u + f(u)=\tau \quad \text{in } \Omega,  \q \tr_Vu=\nu_1$$
and therefore, by Theorem \ref{t:reduced2} (d), 
$\gl\leq \tau^\#(\nu_1)$. 
As $\gl= \tau^\#(\nu_2)$ satisfies the conditions required above, this implies \eqref{nu1,2}.

\vskip 2mm

\noindent\textbf{Step II.}\hskip 2mm  
With $\nu^*$ as before we prove that
\begin{equation}\label{temp4.4}
\tau^\#(\nu^*)=  \tau^\#(0).
\end{equation}

Let $u^\#_0$ be the solution of 
$$-L_V u+ f(u)=\tau^\#(0) \quad \text{in } \Omega,\q \tr_V u=0.$$
Then $u_0^\#$ is a subsolution of \eqref{bvp-tau}. By Theorem \ref{t:reduced2}(d) and \eqref{nu*1}, 
$$u_0^\#\leq u^\#(\tau,\nu) = u^\#(\tau,\nu^*).$$

Let $w :=  u^\#(\tau,\nu^*)-u_0^\#$ so that
$$-L_Vw + f(u^\#(\tau,\nu^*)) - f(u_0^\#)= \tau^\#(\nu^*) - \tau^\#(0) \quad \text{in } \Omega.$$
As $w\geq 0$, by the inverse maximum principle  \cite{Dup-Pon},
$$-(\Gd w )_c = (\tau^\#(\nu^*) - \tau^\#(0))_c\geq 0.$$
Moreover, by Lemma \ref {tau*d},
$$\tau^\#(\nu^*)_d =\tau^\#(0)_d = \tau_d.$$
Hence 
$\tau^\#(\nu^*) \geq \tau^\#(0).$ On the other hand, by step I, $\tau^\#(\nu^*) \leq \tau^\#(0).$ This proves  \eqref{temp4.4}.

(iii) This is a simple consequence of statements (i) and (ii) and Theorem \ref{t:reduced2}(d).
\end{proof}

\remark\label{neg-m} Given a real function $h$ on $\BBR$, denote by $\hat h$ the function given by
$$\hat h(t):=-h(-t) \forevery t\in \BBR.$$

\vskip 2mm

Let $f_n$, $\tau, \nu$ be as in  Theorem \ref{t:reduced2}. If $w_n$ is the solution of the \bvp
$$
-L_V w + \hat f_n(w) = \tau \quad \text{in } \Omega,\q  \tr_V w =\nu,
$$
then $z_n=-w_n$ satisfies 
$$
-L_V z_n + f_n(z_n) = -\tau \quad \text{in } \Omega,\q  \tr_V z_n = -\nu.
$$
Since $\hat f_n$ has the same properties as $f_n$, the sequence $\{w_n\}$ has the same properties as the sequence $\{u_n\}$ in Theorem \ref{t:reduced2}. 

Accordingly, for $(\tau,\nu)\in \GTM_+(\Gw;\Phi_V)\ti \GTM_+(\bdw)$ the reduced measures for $-\tau$ and $-\nu$ and the corresponding reduced couple are given by
\begin{equation}\label{negative-m}\BAL
 (-\tau)^\#:= -(\tau^\#_{\hat f}), &\q (-\nu)^\#:= -(\nu^\#_{\hat f}), \\
(-\tau,-\nu)^\# := -((\tau,&\nu)^\#_{\hat f}) = ((-\tau)^\#,(-\nu)^\#).
\EAL\end{equation}
(The subscript $\hat f$ above indicates that the reduced measure or couple is defined relative to this function.)

In this case $(-\tau)^\#$ and $(-\nu)^\#$ are the smallest good measures dominating $-\tau$ and $-\nu$ respectively.
The relation between the solutions corresponding to the reduced couples is given below
$$
u^\#(-\tau,-\nu)=-u^\#_{\hat f}(\tau,\nu).
$$
\vskip 1mm
Here again the notation $u^\#_{\hat f}$ indicates that the reduced couple and the corresponding solution is defined relative to $\hat f$. 

In view of this remark, \textit{the results of Theorems \ref{t:reduced2} and  \ref{t:reduced3}, with obvious modifications, also apply to couples of negative measures}.

\section{Signed measures}
\begin{theorem} \label{signed-reduced}
	Let $\tau \in \GTM(\Omega;\Phi_V)$ and $\nu \in \GTM(\partial \Omega)$.
Let	$\tau_1, \tau_2 \in \GTM_+(\Omega;\Phi_V)$ and $\nu_1, \nu_2 \in \GTM_+(\partial \Omega)$ be measures \sth
	\begin{equation}\label{tau12}
-\tau_1 \leq \tau \leq \tau_2, \q -\nu_1 \leq \nu \leq \nu_2.
	\end{equation} 
	
	Let $\{f_n\}$ be a sequence of functions as in Theorem \ref{t:reduced2}. Then the \bvp \eqref{bvpn'} has a unique solution $u_n= u_n(\tau,\nu)$.
	
	The following statements hold.
	
	(a) The sequence $\{u_n\} $ is  bounded in $W^{1,p}_{\rm loc}(\Gw)$ for every $p \in [1,\frac{N}{N-1})$. Consequently every subsequence has a limit point in the sense of convergence in $L^p_{loc}(\Gw)$ and convergence a.e. in $\Gw$.
	
	If $\{u_{n_k}\}$ is a subsequence of $\{u_n\}$ converging to  $\tl u$ a.e. then:
		\begin{equation}\label{un-conv}\BAL
(i) &\q	u_{n_k} \to \tl u  \qtxt{in } L^1(\Gw;\Phi_V/\gd), \\
(ii) &\q  f_{n_k}(u_{n_k})\to f(\tl u) \qtxt{a.e. in } \Omega.
\EAL	\end{equation}

	
	(b) The following inequality holds
	\begin{equation} \label{estuf(u)}
	\| \tilde u \|_{L^1(\Omega;\Phi_V/\delta)} + \| f(\tilde u) \|_{L^1(\Omega;\Phi_V)} \leq C \sum_{i=1,2} (\| \tau_i  \|_{\GTM(\Omega;\Phi_V)} + \|\nu_i \|_{\GTM(\partial \Omega)}).
	\end{equation}
	
	(c) There exist $\tilde \tau \in \GTM(\Omega;\Phi_V)$ and $\tilde \nu \in \GTM(\partial \Omega)$ such that
	\begin{equation}\label{tl_tau,nu}
(-\tau_1)^\# \leq \tilde \tau \leq \tau_2^\#, \q (-\nu_1)^\# \leq \tilde \nu \leq \nu_2^\#
	\end{equation} 
	 and 
	$$
	\tilde u + \BBG_V[f(\tilde u)] = \BBG_V[\tilde \tau] + \BBK_V[\tilde \nu] \quad \text{in } \Omega.
	$$
	Thus $\tilde u$ is the solution of the boundary value problem
\begin{equation} \label{bvp_tlu}
	-L_V \tilde u + f(\tilde u)=\tilde \tau \quad \text{in } \Omega,\quad 
	\tr_V \tilde u =\tilde \nu.
	\end{equation}
\end{theorem}
\begin{proof}
Let $v_{2,n}$  denote the solution of the \bvp
\begin{equation}\label{sol_2n}
-L_V v + f_n(v)=\tau_2 \quad \text{in } \Omega, \q \tr_V v=\nu_2,
\end{equation}
and $v_{1,n}$  denote the solution of the \bvp
\begin{equation}\label{sol_1n}
-L_V v + f_n(v)=-\tau_1 \quad \text{in } \Omega, \q \tr_V v=-\nu_1.
\end{equation}		

By Theorem \ref{t:reduced2} and Remark \ref{neg-m}, $v_{2,n}$ and $-v_{1,n}$ are positive and the following inequalities hold
\begin{equation} \label{eq:wn}\BAL
 \| v_{i,n} \|_{L^1(\Omega;\Phi_V/\delta)} + &\| f_n(v_{i,n}) \|_{L^1(\Omega;\Phi_V)}\\
  \leq & C(\| \tau_i \|_{\GTM(\Omega;\Phi_V)} + \| \nu_i \|_{\GTM(\partial \Omega)}),\q i=1,2.
\EAL\end{equation}
Moreover,
\begin{equation*}\BAL
v_{1,n} \uparrow v_1^\#(-\tau,-\nu)=: w_1,&  \q v_{2,n} \downarrow v_2^\#(\tau,\nu)=: w_2 \q \text{a.e. in } \Omega,\\
 f(v_{i}^\#) &\in L^1(\Omega;\Phi_V), \quad i=1,2,
\EAL\end{equation*} 
and there exist
$$ \tau_2^\#, (-\tau_1)^\# \in \GTM(\Omega;\Phi_V), \q \nu_2^\#, (- \nu_1)^\# \in \GTM(\partial \Omega)$$ such that
\begin{equation} \label{eq:vw}\BAL
-L_V w_2 + f(w_2 )&= \tau_2^\# \quad \text{in $\Gw$}, \q &&
\tr_V  w_2 =\nu_2^\#, \\
-L_V w_1 +  f (w_1)&= (-\tau_1)^\# \quad \text{in $\Gw$},
 \q &&\tr_V  w_1 =(-\nu_1)^\#.
\EAL\end{equation}

The monotone convergence of the sequences $\{v_{i,n}\}$ and \eqref{eq:wn} imply
\begin{equation}\label{to_vw}
v_{i,n}\to w_i \qtxt{in }\; L^1(\Gw;\Phi_V/\gd)
\end{equation}
and, by Dini's lemma,
\begin{equation}\label{fn_vw}
f_n(v_{i,n})\to f(w_i) \qtxt{a.e. in } \Omega, \q i=1,2.
\end{equation}
\vskip 2mm
	
(a)	By \eqref{bvpn'} and \eqref{sol_2n}, we have
	\begin{equation}\label{v2n}\BAL
	-L_V(u_n-v_{2,n}) + f_n(u_n)- f_n(v_{2,n})&=\tau-\tau_2 \quad \text{in } \Omega\\
	\tr_V(u_n-v_{2,n})&=\nu-\nu_2.
\EAL	\end{equation}
	By \eqref{tau12}, $\tau-\tau_2\leq 0$ and $\nu-\nu_2\leq 0$. Therefore, using Kato's inequality, we obtain $-L_V (u_n-v_{2,n})_+\leq 0$ and, by Lemma \ref{tr+},  $\tr_V (u_n-v_{2,n})_+=0$. By Lemma \ref{subhar=0}, $(u_n-v_{2,n})_+ =0$, which implies
	\begin{equation}\label{un<}
	u_n\leq v_{2,n} \quad \text{in } \Omega.
	\end{equation}
	
Similarly, by \eqref{bvpn'} and \eqref{sol_1n}, we obtain
	\begin{equation}\label{un>}
v_{1,n}\leq u_n \quad \text{in } \Omega.
	\end{equation}
	(Recall that $v_{1,n}\leq 0$.)
	Thus
	\begin{equation}\label{|u_n|}
	|u_n| \leq \max\{-v_{1,n}, v_{2,n}\} \quad \text{in } \Omega.
	\end{equation} 
This inequality and the monotonicity of $f_n$ imply that
\begin{equation}\label{f|u_n|}
|f_n(u_n)| \leq \max\{f_n(-v_{1,n}), f_n(v_{2,n})\} \quad \text{in } \Omega.
\end{equation}
 Inequalities \eqref{|u_n|}, \eqref {f|u_n|} and \eqref{eq:wn} together with the fact that $V$ is locally bounded in $\Gw$ imply that $\{|V u_n + f_n(u_n)|\}$ is  bounded in $L^1_{\rm loc}(D)$. Consequently $\{u_n\} $ is  bounded in $W^{1,p}_{\rm loc}(\Gw)$ for every $p \in [1,\frac{N}{N-1})$. Therefore $\{u_n\}$ is precompact in  $L^p_{\rm loc}(\Gw)$. Hence there exists a subsequence $\{u_{n_k}\}$) that converges to a function $\tilde u$ in $L_{\rm loc}^p(\Omega)$ and a.e. in $\Omega$. 
 
	By the generalized dominated convergence theorem, the convergence of $\{u_{n_k}\}$ to $\tl u$ a.e. in $\Omega$,  \eqref{|u_n|} and  \eqref{to_vw} imply 
	\eqref{un-conv} (i).
	
	The assumption $f_n^{\pm} \uparrow f^\pm$ and Dini's lemma imply \eqref{un-conv} (ii).
\vskip 2mm

To simplify the presentation, in the remainder of the proof \textit{we
assume that $\{u_n\}$ is a sequence converging a.e. to $\tl u$}. 
\vskip 2mm

	(b) By \eqref{|u_n|}, \eqref{eq:wn} and \eqref{un-conv}(i), we obtain
	$$ \BAL
	\| \tilde u \|_{L^1(\Omega;\Phi_V/\delta)} &= \lim_{n \to \infty} \| u_n \|_{L^1(\Omega;\Phi_V/\delta)}\\
	&\leq C \sum_{i=1,2} (\| \tau_i  \|_{\GTM(\Omega;\Phi_V)} + \|\nu_i \|_{\GTM(\partial \Omega)}).
	\EAL $$

	By Fatou's lemma, \eqref{f|u_n|}, \eqref{eq:wn} and \eqref{un-conv} yield
	$$
	\BAL
	\| f(\tilde u) \|_{L^1(\Omega;\Phi_V)} &\leq \lim_{n \to \infty}\| f_n(u_n) \|_{L^1(\Omega;\Phi_V)} \\
	& \leq \limsup_{n \to \infty}\| \max\{f_n(-v_{1,n}),f_n(v_{2,n})\} \|_{L^1(\Omega;\Phi_V)} \\
	&\leq C \sum_{i=1,2} (\| \tau_i  \|_{\GTM(\Omega;\Phi_V)} + \|\nu_i \|_{\GTM(\partial \Omega)}).
	\EAL
	$$
This proves \eqref{estuf(u)}.
	\vskip 2mm
	
	(c) Inequalities \eqref{un<} and  \eqref{un>} imply $w_1 \leq \tilde u \leq w_2$ a.e. in $\Omega$.

Let $\gz\in C_c^\infty(\Gw)$ and $\gz\geq 0$.  By \eqref{v2n},
\begin{equation}\label{temp6.1}
 -\int_\Gw (v_{2,n}-u_n)L_V\gz \,dx + \int_\Gw (f_n(v_{2,n})-f_n(u_n))\gz \,dx =\int_\Gw \gz  \,d(\tau_2-\tau).
\end{equation}
By \eqref{un<}, \eqref{fn_vw} and Fatou's lemma,
\[ \int_\Gw (f(w_2)- f(\tl u))\gz\,dx  \leq \liminf_{n \to \infty} \int_\Gw (f_n(v_{2,n})-f_n(u_n))\gz \,dx.    \]
Therefore, by \eqref{temp6.1} and \eqref{to_vw},
\begin{equation}\label{temp6.1+}
-\int_\Gw (w_2- \tl u) L_V\gz dx +   \int_\Gw (f(w_2)- f(\tl u))\gz\,dx\leq  \int_\Gw \gz  d(\tau_2-\tau).
\end{equation}

By \eqref{temp6.1}, we have
$$
\left|\int_\Gw (v_{2,n}-u_n)L_V\gz dx\right|\leq  \int_\Gw (f_n(v_{2,n})-f_n(u_n))\gz \,dx +\int_\Gw \gz  d(\tau_2-\tau).
$$
By \eqref{eq:wn}, \eqref{|u_n|} and  \eqref{f|u_n|},
$$\BAL 0 &\leq \int_\Gw (f_n(v_{2,n})-f_n(u_n))\gz \,dx +\int_\Gw \gz  d(\tau_2-\tau)\\
&\leq c \sup(\gz/\Phi_V) \big( \sum_{i=1,2}(\norm{\tau_i}_{\GTM(\Gw;\Phi_V)} +  \norm{\nu_i}_{\GTM(\bdw)} + \norm{\tau}_{\GTM(\Gw;\Phi_V)} \big )\\
&\leq  C\sup (\gz/\Phi_V).
\EAL$$
Since 
$$ \lim_{n \to \infty}\int_\Gw (v_{2,n}-u_n)L_V\gz dx = \int_\Gw (w_2-\tl u)L_V\gz dx,$$
it follows that 
\begin{equation*}
\left|\int_\Gw (w_2-\tl u)L_V\gz dx \right| \leq C\sup (\gz/\Phi_V).
\end{equation*}

Hence, by \eqref{temp6.1+}, there exists
a measure $\gl\leq \tau_2-\tau$ \sth  $\gl \in \GTM(\Gw;\Phi_V)$ and 
$$ -\int_\Gw (w_2- \tl u) L_V\gz \,dx +   \int_\Gw (f(w_2)- f(\tl u))\gz\,dx =\int_\Gw\gz \,d\gl \forevery \gz\in C_c^\infty(\Gw) $$
or equivalently,
\begin{equation}\label{gl}
-L_V  (w_2- \tl u) + f(w_2) -f( \tl u)=\gl \quad \text{in } \Omega.
\end{equation}
Consequently, by \eqref{eq:vw},
\begin{equation}\label{temp6.2}
-L_V \tl u +f(\tl u)=\tl \tau \quad \text{in } \Omega \qtxt{where }\; \tl\tau:= \tau_2^\# -\gl\geq \tau_2^\# - \tau_2 +\tau.  
\end{equation}
\vskip 2mm

Next, by \eqref{bvpn'} and \eqref{sol_1n}, for $\gz\in C_c^\infty(\Gw)$,
$$
-\int_\Gw (u_n-v_{1,n})L_V\gz \,dx + \int_\Gw (f_n(u_n) - f_n(v_{1,n}))\gz \,dx =\int_\Gw \gz  \,d(\tau_1+\tau).
$$
By the same argument as above it follows that there exists a measure $\gl'\leq \tau_1+\tau$ \sth $\gl'\in \GTM(\Gw;\Phi_V)$ and
\begin{equation}\label{gl'}
-L_V  (\tl u-w_1) + f( \tl u) - f(w_1) =\gl' \quad \text{in } \Omega.
\end{equation}
Consequently, by \eqref{eq:vw},
\begin{equation}\label{temp6.2'}
\tl\tau = \gl' + (- \tau_1)^\#\leq \tau_1+\tau + (- \tau_1)^\#.
\end{equation}
\vskip 2mm

Next we show that $\tl\tau$ satisfies \eqref{tl_tau,nu}.
By Lemma \ref{tau*d},
\begin{equation}\label{taui_d}
(\tau_2^\#)_d = (\tau_2)_d, \q ((-\tau_1)^\#)_d = -(\tau_1)_d.
\end{equation}
 Therefore, by \eqref{temp6.2}, $\tl\tau_d\geq \tau_d$ and by  \eqref{temp6.2'}, $\tl\tau_d\leq \tau_d$. Thus
\begin{equation}\label{td=tltd}
\tau_d=\tl\tau_d.
\end{equation}

By \eqref{gl} and \eqref{temp6.2},
$$
 -L_V  (w_2- \tl u) + f(w_2) -f( \tl u)= \tau_2^\# - \tl \tau \quad \text{in } \Omega. 
$$
Since $\tl u$ and $w_2$ are diffuse and $f(0)=0$, it follows that 
\begin{equation}\label{temp6.4}
(-\Gd(w_2-\tl u))_c= (\tau_2^\# - \tl \tau)_c.
\end{equation}
As $w_2-\tl u\geq 0$, by the inverse maximum principle \cite{Dup-Pon},  $(-\Gd(w_2-\tl u))_c\geq 0$.
 Consequently,
$$
(\tau_2^\# - \tl \tau)_c \geq 0.
$$
As $\tau\leq \tau_2$, \eqref{taui_d}, \eqref{td=tltd} yield, 
$$(\tau_2^\#)_d=(\tau_2)_d \geq \tau_d=\tl\tau_d.$$
This inequality and \eqref{temp6.4} imply
\begin{equation}\label{tlt<}
\tl\tau\leq \tau_2^\#.
\end{equation}

Similarly by \eqref{gl'} and \eqref{temp6.2'},
$$ -L_V  (\tl u-w_1) + f( \tl u) - f(w_1) = \tl \tau - (-\tau_1)^\# \quad \text{in } \Omega.$$
Since $\tl u-w_1\geq 0$, another application of the inverse maximum principle yields
$(-\Gd (\tl u-w_1))_c\geq0$ and consequently
\begin{equation}\label{temp6.5'}
(\tl \tau- (-\tau_1)^\#)_c \geq 0.
\end{equation}
As $-\tau_1\leq \tau$,  \eqref{taui_d}, \eqref{td=tltd} imply
$$\tl\tau_d=\tau_d\geq (-\tau_1)_d= ((-\tau_1)^\#)_d.$$ 
This and \eqref{temp6.5'} yield
$$
\tl\tau\geq (-\tau_1)^\#.
$$
Finally, this and \eqref{tlt<} imply \eqref{tl_tau,nu} \wrto $\tl\tau$. 
\vskip 2mm

It remains to show that $\tl u$ has an $L_V$ \btr and that the second inequality in \eqref{tl_tau,nu} holds.

By \eqref{gl} $-L_V(w_2-\tl u)=\mu $ where $\mu\in \GTM(\Gw;\Phi_V)$. Since $w_2 -\tl u \geq 0$, by Lemma \ref{exist_tr_pos}, $w_2 -\tl u$ has an $L_V$ \btr, say  $\gs$, and $w_2-\tl u= \BBG_V[\mu]+ \BBK_V[\gs]$. Obviously $\gs\geq 0$.
Therefore 
$$\tr_V \tl u=\tr_V w_2 -\gs \leq \tr_V w_2 = \nu_2^\#.$$ 

Similarly, starting with \eqref{gl'} we conclude that  there exists
$\gs' \in \GTM(\bdw)$ \sth $\tr_V (\tl u- w_1)=\gs'\geq 0$.
Therefore $$\tr_V \tl u =\gs' + \tr_V w_1 \geq \tr_V w_1 = (-\nu_1)^\#.$$
This completes the proof. 
\end{proof} 

The theorem is complemented by the following consequence of \cite[Corollary 3.7]{MM-note} (see Lemma \ref{exist-subsup}).
\begin{proposition}\label{g.couple}
Let $(\gl_i,\gs_i) \in \GTM(\Gw;\Phi_V) \ti \GTM(\bdw)$, $i=1,2$. Suppose that these are good couples \wrto \eqref{bvp-gen} and that $(\gl_1,\gs_1)\prec (\gl_2,\gs_2)$. Then, every couple $(\gl,\gs)$ \sth
\begin{equation}\label{g.interval}
(\gl_1,\gs_1)\prec (\gl,\gs)\prec (\gl_2,\gs_2)
\end{equation}
is a good couple.
\end{proposition}

\proof Let $v_i$ be the solution corresponding to the couple $(\gl_i,\gs_i)$, $i=1,2$ and let $(\gl,\gs)$ be as in \eqref{g.interval}. Then $v_2$ is a supersolution and $v_1$ a subsolution of equation $-L_V u +f(u) = \gl$ and $\tr_V v_1 \leq \gs\leq \tr_V v_2$.
Therefore the stated result is a consequence of Lemma \ref{exist-subsup}).
\qed \\ 

\noindent\textbf{Remark.} As mentioned before, Theorem \ref{signed-reduced} and Proposition \ref{g.couple} imply Theorem \ref{th:main2}. However we emphasize that, in contrast to Theorem \ref{signed-reduced}, in   Proposition \ref{g.couple} $(\gl_i,\gs_i)$ may be couples of signed measures.\\

\begin{proposition} \label{goodcouple1}
In addition to the assumptions of Theorem \ref{signed-reduced}, assume that $(\tau_2,\nu_2)$ and $(-\tau_1, -\nu_1)$ are good couples. 

Suppose that $(-\tau_1, -\nu_1)\prec (\tau,\nu) \prec (\tau_2,\nu_2)$.
 (By the previous result, $(\tau,\nu)$ is a good couple.) Let $u$ be the solution of problem \eqref{bvp-f} and
let $u_n$ denote the solution of the `approximating' problem
$$-L_V u+f_n( u)= \tau \qtxt{in } \Gw, \q \tr_Vu=\nu.$$

If $\Phi_V$ satisfies the additional condition
\begin{equation} \label{C2}  \int_{\Gs_\gb} \Phi_V^2/\gd\; dS\to 0 \qtxt{as }\gb\to 0 
\end{equation}
 then $u_n\to u$, i.e. $\tl u=u$.                

	
\end{proposition}

\begin{proof} We use the notation in the proof of Theorem \ref{signed-reduced}.
	
		Let $v_{1,n}$ and $v_{2,n}$ be the solutions of \eqref{sol_1n} and \eqref{sol_2n}. Then
	\begin{equation}\label{un,vn}
	v_{1,n}\leq u_n\leq v_{2,n} \quad \text{in } \Omega.
	\end{equation} 
	The sequences $\{v_{i,n}\}$, $i=1,2$ satisfy \eqref{to_vw} and \eqref{fn_vw}. In addition, by Proposition \ref{fntof} (see Appendix), 
	 \begin{equation}
	f_n(v_{i,n}) \to f(w_i) \qtxt{in }\; L^1(\Gw;\Phi_V).
	 \end{equation}

	 By \eqref{un,vn} and the monotonicity of $f_n$,
	 $$f_n(v_{1,n}) \leq f_n(u_n) \leq f_n(v_{2,n}) \quad \text{in } \Omega.$$
	 Therefore, taking a subsequence for which \eqref {un-conv} holds, the (generalized) dominated convergence theorem implies $f_n(u_n) \to f(\tl u)$ in $L^1(\Gw;\Phi_V)$ as $n \to \infty$.
Hence, by Theorem \ref	{G-I}, 
$$ \BBG_V[f_n(u_n)] \to \BBG_V[f(\tl u)] \qtxt{in }\; L^1(\Gw;\Phi_V/\gd). $$
By \eqref {un-conv}, $u_n\to \tl u$ in $L^1(\Gw;\Phi_V/\gd)$. As
$$u_n+ \BBG_V[f_n(u_n)] = \BBG_V[\tau] + \BBK_V[\nu] \quad \text{in } \Omega$$
we conclude that
$$\tl u + \BBG_V[f(\tl u)] = \BBG_V[\tau] + \BBK_V[\nu] \quad \text{in } \Omega.$$
Thus $\tl u$ is a solution of \eqref{bvp-f}. By uniqueness (Lemma \ref{sub-sup}), $\tl u =u$. 
\end{proof}


\begin{proof}[\textbf{Proof of Theorem \ref{th:main3}}]\hskip 3mm
Since $f$ vanishes on $(-\infty,0]$, if $w$ is a  real function on $\Gw$ then
\begin{equation}\label{fw+}
f(w)=f(w_+)+f(-w_-) =f(w_+).
\end{equation}

Suppose that $(\tau, \nu)$ is a good couple, i.e. \eqref{bvp-tau} has a solution $u$ \sth $f(u)\in L^1(\Gw;\Phi_V)$. 

Let $u_n$ and $\tl u$ be as in Theorem \ref{signed-reduced}. In view of \eqref{fw+}, $\{u_n\}$ is decreasing (by Lemma \ref{sub-sup}) and $\tl u=\lim u_n$.

Let $w$ be a subsolution of \eqref{bvp-tau} \sth $f(w)\in L^1(\Gw;\Phi_V)$. Then
$$-L_Vw + f_n(w) \leq -L_Vw + f(w_+) \leq \tau \quad \text{in } \Omega,\q \tr_V w\leq \nu.$$
As $u_n$ satisfies \eqref{bvpn'}, Lemma \ref{sub-sup} implies that
 $w\leq u_n$. Thus $w\leq \tl u$ and, in particular,
\begin{equation}\label {u<tlu}
u\leq \tl u \quad \text{in } \Omega.
\end{equation}

Let $v:=\tl u-u$. Then $v\geq 0$, $\Gd v$ is a measure and therefore, by the inverse maximum principle, 
$$(-\Gd v)_c = (\tl\tau -\tau)_c\geq 0.$$
As in the proof of  Theorem \ref{signed-reduced} (see \eqref{td=tltd}) $\tau_d=\tl\tau_d$. Therefore
$$
\tau\leq \tl\tau.
$$
In addition, by \eqref{u<tlu},
$$
\nu=\tr_Vu \leq  \tr_V\tl u=\tilde \nu.
$$
Thus
\begin{equation}\label{tau_prec}
 (\tau,\nu) \prec (\tl \tau, \tl \nu).
\end{equation}

If 
$(\tau_1,\nu_1):= (\tau_-,\nu_-)$ and $(\tau_2,\nu_2):= (\tau_+, \nu_+)$ then
 $\tau,\nu$ satisfy \eqref{tau12} and by  Theorem \ref{signed-reduced}(c),
\begin{equation}\label{tltau_prec}
(\tl\tau, \tl\nu) \prec (\tau_+,\nu_+)^\#.
\end{equation}
Hence, by \eqref{tau_prec}, we obtain \eqref{good_couple}.

  Conversely, assume that $(\tau,\nu)$ satisfies \eqref{good_couple}. Recall that every couple of negative measures is good relative to $f$. Therefore, by Proposition \ref{g.couple}, the relation
$$-(\tau_-,\nu_-) \prec (\tau, \nu) \prec (\tau_+,\nu_+)^\#$$
implies that $(\tau,\nu)$ is a good couple.

The last assertion of the theorem is obvious. 
\end{proof}

\vspace{0.3cm}

\noindent\textbf{Acknowledgement.}\hskip 2mm The research of M. Bhakta is partially supported by the DST Swarnajaynti fellowship (SB/SJF/2021-22/09). The research of P.-T. Nguyen was supported by Czech Science Foundation, Project GA22-17403S.

\vspace{0.3cm}

\noindent\textbf{Conflict of interest.} The authors declare that they have no conflicts of interest to this work.

\appendix \section{~} \label{Appendix}
\setcounter{equation}{0}
\renewcommand{\theequation}{a.\arabic{equation}}

We prove an auxilliary result that is used in the proof of Proposition \ref{goodcouple1}.


\begin{proposition} \label{fntof} Assume \eqref{A1}, \eqref{A2}, \eqref{C1} and \eqref{C2}  hold.
Suppose that $(\tau,\nu)\in  \GTM_+(\Omega;\Phi_V)\ti \GTM_+(\partial \Omega)$ is a good couple of measures. Let $u$ be the corresponding solution of problem \eqref{bvp-tau} and $u_n$ be the solution of \eqref{bvpn'}. Then 
$$u_n \to u \qtxt{in }\;L^1(\Omega;\Phi_V/\delta), \q f_n(u_n) \to f(u) \qtxt{in }\;L^1(\Omega;\Phi_V).$$
\end{proposition}

The proof is based on the following lemma that was established in \cite{MN2} for a more restricted class of potentials. 

\begin{lemma}\label{tau-PhiV} Assume \eqref{A1}, \eqref{A2}, \eqref{C1} and \eqref{C2}  hold.
	Let $\tau\in \GTM(\Gw;\Phi_V)$ and let $\gl_V$ be the eigenvalue of $-L_V$ corresponding to $\Phi_V$. Then
	\begin{equation}\label{tau-vgf1}
	\lambda_V \int_{\Omega} \BBG_{V}[\tau] \Phi_{V}\,dx = -\int_{\Omega} \BBG_{V}[\tau] L_V \Phi_V \,dx =\int_{\Omega} \Phi_{V}\,d\tau.
	\end{equation}
\end{lemma}
\begin{proof}  By linearity, we may assume that $\tau \geq 0$. 
	For $\gb>0$, put
	$$I_\gb(\tau):=-\gl_V \int_{D_\gb}\BBG_{V}[\tau] \Phi_{V} \,dx + \int_{D_\gb} \Phi_{V}\,d\tau,$$
	where $D_\gb=\{x\in \Gw: \gd(x)>\gb\}$. To prove \eqref{tau-vgf1} we show that
	\begin{equation}\label{Igb=0}
	\lim_{\gb\to0} I_\gb(\tau)=0.
	\end{equation}		
		
		By Theorem \ref{G-I} (ii), 
	$$I_0(\tau)\leq C\int_\Gw \Phi_V d\tau.$$
	Given $\ge>0$ we choose $0<\gg$ sufficiently small so that, for $\tau_\gg:=\tau \chr{\Gw_{\gg}}$,
\begin{equation}\label{gb_gg}
I_0(\tau_\gg)<\ge.
\end{equation}

Therefore it is sufficient to prove 
$$\lim_{\gb\to 0} I_\gb(\tau-\tau_\gg)= 0.$$
Thus it is sufficient to prove \eqref{Igb=0} when $\tau$ has compact support in $\Gw$.


	Let $\gb_\tau=\rec{2} \dist(\supp\tau, \bdw)$ and let $\gb\in (0, \gb_\tau)$.
	Applying Green's theorem in $D_\gb$, we obtain
	$$ \begin{aligned} \gl_V&\int_{D_\gb}\BBG_{V}[\tau] \Phi_{V} \,dx= - \int_{D_\gb}\BBG_{V}[\tau] L_{V}\Phi_{V} \,dx  \\
	=&\int_{D_\gb}\Phi_{V} \, d\tau
	+\int_{\Sigma_\gb}\frac{\partial \BBG_{V}[\tau]}{\partial\mathbf n}\Phi_{V}\, dS(x)
	-\int_{\Sigma_\gb}\frac{\partial \Phi_{V}}{\partial \mathbf n}  \BBG_{V}[\tau]\, dS(x).
	\end{aligned} $$
	Thus
	\begin{equation}\label{tau-vgf3}
	I(\gb)= -\int_{\Sigma_\gb}\frac{\partial \BBG_{V}[\tau]}{\partial \mathbf n}\Phi_{V}\,dS(x)
	+\int_{\Sigma_\gb}\frac{\partial \Phi_{V}}{\partial \mathbf n}  \BBG_{V}[\tau]\,dS(x).
	\end{equation}
	
	Note that
	$$G_V(x,y) \sim \Phi_V(x) \forevery (x,y) \in \Gw_{\gb_\tau} \ti \supp \tau.$$
	 Therefore
	\begin{equation} \label{Gtausim} \BBG_V[\tau](x) = \int_{\Omega} G_V(x,y) \, d\tau(y) \sim \Phi_V(x), \quad \forall x \in \Sigma_\beta.
	\end{equation}
	By interior elliptic estimates, 
	for every $x\in \Gs_\gb$,
	$$\left |\frac{\partial \Phi_{V}}{\partial \mathbf n}(x) \right |\leq C \sup_{|\gx-x|<\gb/4}  \Phi_{V}(\gx)\gb^{-1}.$$ Therefore by Harnack's inequality, we deduce
	$$\left|  \frac{\partial \Phi_{V}}{\partial \mathbf n}(x) \right| \leq C \Phi_{V}(x)\gb^{-1} \quad \forall x\in \Sigma_\gb.$$
			Hence, by \eqref{Gtausim} and assumption \eqref{C2},
		$$
		\lim_{\gb\to 0}  \int_{\Sigma_\gb}\frac{\partial \Phi_{V}}{\partial\mathbf n}  \BBG_{V}[\tau] \, dS(x) =0.
		$$
	
	In $D_{\gb_\tau}$: $G_V[\tau]$ is $L_V$ harmonic and $G_V[\tau] \sim \Phi_V$. Therefore the same argument as above yields,	
	\begin{equation}\label{tau-vgf5}
	\lim_{\gb\to 0}  \int_{\Sigma_\gb}\frac{\partial \BBG_{V}[\tau]}{\partial\mathbf n}  \Phi_{V} \, dS(x) =0.
	\end{equation}
	Combining \eqref{tau-vgf3} -- \eqref{tau-vgf5}, we obtain \eqref{Igb=0} for measures $\tau$ with compact support. In view of previous remarks, this implies \eqref{Igb=0} for any  measure $\tau\in \GTM(\Gw;\Phi_V)$. This in turn implies \eqref{tau-vgf1}.
\end{proof}
\vskip 2mm

\begin{proof}[\textbf{Proof of Proposition \ref{fntof}}] \hskip 3mm
	By Theorem \ref{t:reduced2}(d), $u=u^\#$. From the proof of Theorem \ref{t:reduced2}, $u_n\geq 0$ satisfies \eqref{eq-un},  $u_n \downarrow u^\#=u$ and $f_n(u_n)\to f(u)$ a.e. in $\Omega$. By \eqref{eq-un}, $ u_n\leq \BBG_V[\tau]+ \BBK_V[\nu] \in L^1(\Omega;\Phi_V/\delta)$. Therefore, by the dominated convergence theorem, 
	$$u_n\to u \qtxt{in }\;L^1(\Omega;\Phi_V/\delta).$$
	
	
	By Lemma \ref{tau-PhiV} with $\tau$ replaced by $f_n(u_n)$ (recall that $f_n$ is a bounded function) we have,
	\begin{equation}\label{e:f_n}
	 \int_{\Omega} f_n(u_n)\Phi_V\,dx = \lambda_V \int_{\Omega} \BBG_V[f_n(u_n)]\Phi_V\, dx.
	 \end{equation}
	Since $u_n$ is the solution of \eqref{bvpn'} it satisfies $$u_n+\BBG_V[f_n(u_n)]= \BBG_V[\tau] + \BBK_V[\nu].$$
	Multiplying this equality by $\gl_V\Phi_V$ and using \eqref{e:f_n} we obtain
$$\BAL
&\int_{\Omega} f_n(u_n)\Phi_V\,dx = \lambda_V \int_{\Omega} \BBG_V[f_n(u_n)]\Phi_V\, dx \\
=-& \lambda_V \int_{\Omega} u_n \Phi_V\, dx  +  \lambda_V \int_{\Omega} \BBG_V[\tau] \Phi_V\, dx + \lambda_V \int_{\Omega} \BBK_V[\nu] \Phi_V\, dx.
\EAL $$
Hence, 
\begin{equation}\label{e:fn(un)}
	\BAL
&\lim_{n \to \infty}\int_{\Omega} f_n(u_n)\Phi_V \,dx =\\
&- \lambda_V \int_{\Omega} u \Phi_V \,dx  +  \lambda_V \int_{\Omega} \BBG_V[\tau] \Phi_V \,dx + \lambda_V \int_{\Omega} \BBK_V[\nu] \Phi_V \,dx.
\EAL\end{equation}

Since $u$ is the solution of \eqref{bvp-tau}, $f(u)\in L^1(\Gw;\Phi_V)$ and $u+\BBG_V[f(u)]= \BBG_V[\tau] + \BBK_V[\nu].$ In addition, by Lemma \ref{tau-PhiV},
$$  \int_{\Omega} f(u)\Phi_V\,dx = \lambda_V \int_{\Omega} \BBG_V[f(u)]\Phi_V\, dx.$$
 Therefore, as before,
\begin{equation}\label{e:f(u)}\BAL
& \int_{\Omega} f(u)\Phi_V\,dx\\
&=- \lambda_V \int_{\Omega} u \Phi_V\, dx  +  \lambda_V \int_{\Omega} \BBG_V[\tau] \Phi_V\, dx + \lambda_V \int_{\Omega} \BBK_V[\nu] \Phi_V\, dx.
\EAL\end{equation}

By \eqref{e:fn(un)} and \eqref{e:f(u)}, $\norm{f_n(u_n)}_{L^1(\Omega;\Phi_V)} \to \norm{f(u)}_{L^1(\Omega;\Phi_V)}$.
As $f_n(u_n) \to f(u)$ a.e. in $\Omega$, it follows that $$f_n(u_n) \to f(u) \qtxt{in }\; L^1(\Omega;\Phi_V).$$
The proof is complete.
\end{proof}

\end{document}